\documentclass[a4paper,reqno]{amsart}

\usepackage[foot]{amsaddr}
\usepackage{graphicx,enumerate,nicefrac,color,bm,cite}
\usepackage{amsmath, amssymb, amstext, amsfonts}
\usepackage{subfig}
\usepackage{pgfplots}

\usepackage{algorithm,algpseudocode,tabularx}

\makeatletter
\newcommand{\multiline}[1]{%
  \begin{tabularx}{\dimexpr\linewidth-\ALG@thistlm}[t]{@{}X@{}}
    #1
  \end{tabularx}
}
\makeatother

\newcommand{\norm}[1]{\left\|#1\right\|}
\newcommand{\operator}[1]{\mathsf{#1}}

\newcommand{\F}{\operator{F}}   

\renewcommand{\H}{\operator{H}}
\newcommand{\G}{\operator{G}}

\newcommand{\dt}{\,\mathsf{d} t}

\newcommand{\x}{\bm{\mathsf{x}}}

\newcommand{\dx}{\,\mathsf{d} \x}
\newcommand{\dprod}[1]{\left<#1\right>}

\newtheorem{theorem}{Theorem}[section]

\theoremstyle{definition}
\newtheorem{remark}[theorem]{Remark}
\newtheorem{experiment}{Experiment}[section]

\title[Adaptive damped Newton method]{An adaptive damped Newton method for strongly monotone and Lipschitz continuous operator equations}

\author{Pascal Heid}
\email{pascal.heid@ma.tum.de}
\address{Department of Mathematics, Technical University of Munich, Boltzmannstr.~3, 85748 Garching bei München, Germany \& Munich Center for Machine Learning (MCML)}



\begin{document}

\begin{abstract}
We will consider the damped Newton method for strongly monotone and Lipschitz continuous operator equations in a variational setting. We will provide a very accessible justification why the undamped Newton method performs better than its damped counterparts in a vicinity of a solution. Moreover, in the given setting, an adaptive step-size strategy will be presented, which guarantees the global convergence and favours an undamped update if admissible.
\end{abstract}

\maketitle

\section{Introduction} \label{sec:introduction}
In these notes, we focus on the damped Newton method for strongly monotone and Lipschitz continuous operator equations on a real-valued Hilbert space $X$ with inner product $(\cdot,\cdot)_X$ and induced norm $\norm{\cdot}_X$. Specifically, given a nonlinear operator $\F:X \to X^\star$, we focus on the equation
\begin{align} \label{eq:F}
u \in X: \qquad \F(u)=0 \quad \text{in} \ X^\star;
\end{align} 
here, $X^\star$ denotes the dual space of $X$. In weak form, this problem is given by
\begin{align} \label{eq:weak}
u \in X: \qquad \dprod{\F(u),v}=0 \qquad \text{for all} \ v \in X,
\end{align}
where $\dprod{\cdot,\cdot}$ signifies the duality pairing in $X^\star \times X$. Throughout this work, we will impose the following structural assumptions on the nonlinear operator $\F: X \to X^\star$:
\begin{enumerate}
\item[(F1)] The operator $\F$ is Lipschitz continuous; i.e., there exists a constant $L >0$ such that 
\begin{align} \label{eq:lipschitz}
|\dprod{\F(u)-\F(v),w}| \leq L \norm{u-v}_X \norm{w}_X \qquad \text{for all} \ u,v,w \in X;
\end{align}
\item[(F2)] The operator $\F$ is strongly monotone; i.e., there exists a constant $\nu>0$ such that
\begin{align}
\dprod{\F(u)-\F(v),u-v} \geq \nu \norm{u-v}_X^2.
\end{align}
\end{enumerate} 
Under those conditions, it is well-known that the operator equation~\eqref{eq:F} has a unique solution $u^\star \in X$; see, e.g., \cite[\S3.3]{Necas:86} or~\cite[\S25.4]{Zeidler:90}. However, since $\F$ is assumed to be nonlinear, it is in general not feasible to solve for $x^\star \in X$. As a remedy, one may apply an iteration scheme to obtain an approximation $u^n$ of $u^\star$. The most famous method to approximate iteratively a root of a nonlinear operator $\F$ is the \emph{Newton method}, which is defined as follows: For a given initial guess $u^0 \in X$, we define recursively
\begin{align} \label{eq:Newton}
u^{n+1}=u^n - \F'(u^n)^{-1}\F(u^n), \qquad n \in \mathbb{N},
\end{align}  
where $\F'(u^n):X \to X^\star$ denotes the Gateaux derivative of $\F$ at $u^n \in X$. The main advantage of the Newton method over other iterative linearisation schemes is the local quadratic convergence rate. However, in many cases, the Newton scheme requires an initial guess close to a solution of the operator equation. To improve the lack of global convergence, one may introduce a damping parameter. Then, the \emph{damped} Newton method is given by
\begin{align} \label{eq:dampedNewton} 
u^{n+1}=u^n-\delta(u^n)\F'(u^n)^{-1}\F(u^n),
\end{align}
where $\delta(u^n)>0$ is a damping parameter, which may depend on the given iterate. Often, it is not clear how to choose the damping parameter optimally. Moreover, choices that guarantee the global convergence a priori may often entail an inferior convergence rate compared to the undamped scheme~\eqref{eq:Newton}. For an extensive overview of Newton methods for nonlinear problems we refer the interested reader to~\cite{Deuflhard:04}, and the references therein.

The motivation of the present work is to suggest an adaptive step-size selection, which yields, in the given setting, the global convergence and matches the performance of the classical Newton scheme locally. Moreover, we highlight in a variational setting that the undamped scheme is indeed optimal close to a solution.

\section{Convergence of the damped Newton method}

In many problems of scientific interest, $X$ is an infinite dimensional space, but numerical computations are carried out on finite dimensional subspaces, e.g., finite element spaces. We note, however, that our analysis equally applies to closed subspaces of $X$ --- infinite and finite dimensional ones. In order to guarantee the global convergence of the damped Newton method, we need to impose some further assumptions on $\F$:
\begin{enumerate}
\item[(F3)] The operator $\F$ is Gateaux differentiable. Moreover, $\F'$ is uniformly coercive and bounded in the sense that, for any given $u \in X$,
\begin{align} \label{eq:Fpcoercive}
\dprod{\F'(u)v,v} \geq \alpha_{\F'} \norm{v}_X^2 \qquad \text{for all} \ v \in X
\end{align}
and
\begin{align} \label{eq:Fpbounded}
\dprod{\F'(u)v,w} \leq \beta_{\F'} \norm{v}_X \norm{w}_X \qquad \text{for all} \ v,w \in X,
\end{align}
where $\alpha_{\F'},\beta_{\F'}>0$ are independent of $u$;
\item[(F4)] There exists a Gateaux differentiable functional $\G:X \to \mathbb{R}$ such that $\G'(u)=\F'(u)u$ in $X^\star$ for any $u \in X$, and $\G':X \to X^\star$ is continuous with respect to the weak topology in $X^\star$;
\item[(F5)] $\F$ is a potential operator; i.e., there exists a Gateaux differentiable functional $\H:X \to \mathbb{R}$ with $\H'=\F$.
\end{enumerate}
Then, the following convergence result holds.

\begin{theorem}[{\hspace{1sp}\cite[Theorem~2.6]{HeidWihler:2020a}}] \label{thm:convergence}
Assume {\rm (F1)--(F5)} and that $\delta:X \to [\delta_{\min},\delta_{\max}]$ for some $0<\delta_{\min}\leq \delta_{\max}<\nicefrac{2 \alpha_{\F'}}{L}$. Then, the damped Newton method~\eqref{eq:dampedNewton} converges to the unique solution $u^\star \in X$ of~\eqref{eq:F}.
\end{theorem}

In many cases we have that $\nicefrac{2 \alpha_{\F'}}{L}<1$. This, however, excludes the choice $\delta(u^n)=1$, which yields the local quadratic convergence rate. By examining the proof of~\cite[Theorem~2.6]{HeidWihler:2020a}, one realises that the specific upper bound on the damping parameter was imposed to guarantee the existence of a constant $C_{\H}>0$ such that
\begin{align} \label{eq:Hdecay}
\H(u^n)-\H(u^{n+1}) \geq C_\H \norm{u^n-u^{n+1}}_X^2 \qquad \text{for all} \ n \in \mathbb{N};
\end{align}  
we also refer to the closely related work~\cite{HeidPraetoriusWihler:2020}, especially Eq.~(2.9) in that reference. Indeed, in the proof of~\cite[Theorem 2.6]{HeidWihler:2020a} it is shown that
\begin{align} \label{eq:help}
\H(u^n)-\H(u^{n+1}) \geq \left(\frac{\alpha_{\F'}}{\delta(u^n)}-\frac{L}{2}\right) \norm{u^n-u^{n+1}}_X^2.
\end{align}
Consequently, given that $\delta(u^n) \leq \delta_{\max} < \nicefrac{2 \alpha_{\F'}}{L}$, we have
\begin{align} \label{eq:CH}
\H(u^n)-\H(u^{n+1}) \geq \left(\frac{\alpha_{\F'}}{\delta_{\max}}-\frac{L}{2}\right)\norm{u^n-u^{n+1}}_X^2=C_\H \norm{u^n-u^{n+1}}_X^2,
\end{align}
where $C_\H:=\left(\frac{\alpha_{\F'}}{\delta_{\max}}-\frac{L}{2}\right)>0$. By no means, however, is $\delta(u^n) \leq \delta_{\max} < \nicefrac{2 \alpha_{\F'}}{L}$ a necessary condition for~\eqref{eq:Hdecay}. The following generalised convergence theorem remains true, which can be easily verified by following along the lines of the proof~\cite[Theorem 2.6]{HeidWihler:2020a}.

\begin{theorem} \label{thm:generalisedconvergence}
Given {\rm (F1)--(F5)}. For any damping strategy that guarantees $\delta(u^n) \in [\delta_{\min},\delta_{\max}]$ and ~\eqref{eq:Hdecay} for all $n \in \mathbb{N}$, where $0<\delta_{\min} \leq \delta_{\max}<\infty$ and $C_{\H}>0$ are any given constants independent of $n$, we have that $u^n \to u^\star$ as $n \to \infty$. 
\end{theorem}

\section{Local optimality of the classical Newton scheme}
Before we introduce our step-size strategy, which is borrowed from~\cite[Section 3.1]{HeidWihlerm2an:2022}, we first remark that the unique solution of the operator equation~\eqref{eq:F} is equally the unique minimiser of the potential $\H$, see~\cite[Theorem 25.F]{Zeidler:90}. Furthermore, the damped Newton method can be reformulated as
\begin{align} \label{eq:dampedNewton2}
u^{n+1}=u^n-\delta^n \rho^n,
\end{align}
where $\delta^n$ is the damping parameter and $\rho^n=\F'(u^n)^{-1}\F(u^n)$ is the undamped update at the given iterate. Equivalently, we have that
\begin{align} \label{eq:diffrho}
u^{n+1}-u^n=-\delta^n \rho^n. 
\end{align}

Since we want to obtain the unique minimiser $u^\star \in X$ of $\H$, a sensible strategy is to choose the step-size $\delta^n$ in~\eqref{eq:dampedNewton2} in such a way that, for given $u^n \in X$, the difference
\begin{align}
\H(u^{n+1})-\H(u^{n})
\end{align}
is minimised. Indeed, this corresponds to a maximal decay of the potential at step $n+1$. We will now follow along the lines of~\cite[Section 3.1]{HeidWihlerm2an:2022}, where the strategy was originally developed, to the best of our knowledge, as a possible choice of the damping parameter of a modified Ka\v{c}anov scheme. Let us recall that $\H$ is the potential of $\F$, and thus $\H'=\F$. Consequently, thanks to the fundamental theorem of calculus, we have that
\begin{align} \label{eq:Hdiff}
\H(u^{n+1})-\H(u^{n})=\int_0^1 \dprod{\F(u^{n}+t(u^{n+1}-u^n)),u^{n+1}-u^n} \dt.
\end{align}
Let $\psi(t):=\dprod{\F(u^{n}+t(u^{n+1}-u^n)),u^{n+1}-u^n}$ be the integrand of the right-hand side above. In view of~\eqref{eq:diffrho} and the linearity of the dual product, the integrand $\psi$ can be rewritten as
\begin{align} \label{eq:integrand}
\psi(t)=-\delta^n \dprod{\F(u^n-t \delta^n \rho^n),\rho^n}.
\end{align}
In order to minimise~\eqref{eq:Hdiff}, we will employ a first order Taylor approximation to \eqref{eq:integrand} at $t=0$. Specifically, a straightforward calculation reveals that
\begin{align} \label{eq:psitaylor}
\psi(t) \approx -\delta^n \dprod{\F(u^n),\rho^n}+t(\delta^n)^2 \dprod{\F'(u^n)\rho^n,\rho^n}.
\end{align} 
Plugging~\eqref{eq:psitaylor} into \eqref{eq:Hdiff} and, subsequently, integrating from zero to one yields
\begin{align} \label{eq:approxHdiff}
\H(u^{n+1})-\H(u^{n}) \approx \frac{1}{2}(\delta^n)^2 \dprod{\F'(u^n)\rho^n,\rho^n}-\delta^n \dprod{\F(u^n),\rho^n}.
\end{align}
Then, since $\F'(u^n)$ is coercive, it follows immediately that the right-hand side of~\eqref{eq:approxHdiff} is minimised for
\begin{align} \label{eq:optimalstep}
\delta^n=\frac{\dprod{\F(u^n),\rho^n}}{\dprod{\F'(u^n)\rho^n,\rho^n}}=\frac{\dprod{\F(u^n),\rho^n}}{\dprod{\F(u^n),\rho^n}}=1;
\end{align}
here, we used that $\F'(u^n)\rho^n=\F'(u^n)\F'(u^n)^{-1}\F(u^n)=\F(u^n)$. Furthermore, the approximations~\eqref{eq:psitaylor} and~\eqref{eq:approxHdiff}, respectively, become more accurate the smaller $\norm{\rho^n}_X$ gets, and thus especially in a neighbourhood of a solution. In particular, this means that $\delta \equiv 1$ is indeed the optimal local damping parameter in a vicinity of the solution.

\section{Adaptive damped Newton algorithm}

Now we shall present our adaptive damped Newton method, cf.~Algorithm~\ref{alg:AdaptivedampedNewton}, which guarantees the global convergence and favours the step-size $\delta^n=1$, if admissible, leading to local quadratic convergence rate. This algorithm is closely related to~\cite[Algorithm~1]{HeidWihlerm2an:2022}. 

\begin{algorithm}
\caption{Adaptive damped Newton method}
\label{alg:AdaptivedampedNewton}
\begin{flushleft} 
\textbf{Input:} Initial guess $u^0 \in X$, correction factor $\sigma \in (0,1)$, and $\theta \in (0,0.5]$. 
\end{flushleft}
\begin{algorithmic}[1]
\For {$n=0,1,2,\dotsc$}
\State Solve the linear problem $\F'(u^n)\rho^n=\F(u^n)$ for $\rho^n\in X$.
\State Set $\delta^n=1$. 
\Repeat  	
\State Compute $u^{n+1}=u^{n}-\delta^n\rho^n$.
\State Update $\delta^n \gets  \max \left\{\sigma \delta^n,\nicefrac{\alpha_{\F'}}{L}\right\}$.
\Until {$\H(u^{n})-\H(u^{n+1}) \ge  \theta \min\left\{\alpha_{\F'},L\right\}\norm{u^{n}-u^{n+1}}_X^2$}.
\EndFor
\end{algorithmic}
\end{algorithm}

\begin{theorem}
Given the assumptions {\rm (F1)--(F5)}, the sequence $\{u^n\}$ generated by Algorithm~\ref{alg:AdaptivedampedNewton} converges strongly to the unique solution $u^\star \in X$ of the operator equation~\eqref{eq:F}.
\end{theorem}

\begin{proof}
Let $C_\H:=\theta \min\{\alpha_{\F'},L\}$. Moreover, without loss of generality we may assume that $\nicefrac{\alpha_{\F'}}{L} \leq 1$. In particular, we have that $0<\delta_{\min}:=\nicefrac{\alpha_{\F'}}{L} \leq 1=:\delta_{\max}$. Thanks to Theorem~\ref{thm:generalisedconvergence}, it only remains to prove~\eqref{eq:Hdecay}. By line 7 of Algorithm~\ref{alg:AdaptivedampedNewton}, this amounts to verify that the repeat-loop (lines 4--7) in Algorithm~\ref{alg:AdaptivedampedNewton} terminates for each $n \in \mathbb{N}$. By contradiction, assume that this was not the case; i.e.~there exists $n \in \mathbb{N}$ such that the statement in line 7 is never satisfied. For this specific $n \in \mathbb{N}$, we have after finitely many steps that $\delta^n=\nicefrac{\alpha_{\F'}}{L}$. Recalling~\eqref{eq:help}, we find that 
\begin{align*}
\H(u^n)-\H(u^{n+1}) &\geq \left(\frac{\alpha_{\F'}}{\delta^n}-\frac{L}{2}\right)\norm{u^{n}-u^{n+1}}_X^2 \\
&=\frac{L}{2} \norm{u^{n}-u^{n+1}}_X^2 \\
&\geq C_{\H} \norm{u^n-u^{n+1}}_X^2;
\end{align*}
i.e., the repeat-loop terminates, which yields the desired contradiction.
\end{proof}

\begin{remark}
Of course, Algorithm~\ref{alg:AdaptivedampedNewton} could be further improved by taking into account the accepted damping parameter from the previous step. However, for simplicity of the presentation, we will stick to our current algorithm.
\end{remark}

\section{Application to quasilinear elliptic diffusion models} 

In the following, let $X:=H^1_0(\Omega)$ denote the standard Sobolev space of $H^1$-functions on $\Omega$ with zero trace along the boundary $\Gamma:=\partial \Omega$, where $\Omega \subset \mathbb{R}^d$, $d \in \mathbb{N}$, is an open and bounded domain with Lipschitz boundary. The inner product and norm on $X$ are defined by $(u,v)_X:=\int_\Omega \nabla u \cdot \nabla v \dx$ and $\norm{u}_X^2=\int_\Omega |\nabla u|^2 \dx$, respectively. We will consider the quasilinear elliptic partial differential equation
\begin{align} \label{eq:modelproblem}
u \in X: \qquad \F(u):=-\nabla \cdot \left\{\mu\left(|\nabla u|^2\right)\nabla u\right\}-g=0 \qquad \text{in} \ X^\star;
\end{align}
here, $\mu:\mathbb{R}_{\geq 0} \to \mathbb{R}_{\geq 0}$ is a diffusion coefficient and $g \in H^{-1}(\Omega) = X^\star$ a given source function. Models of the form~\eqref{eq:modelproblem} are widely applied in physics, for instance in plasticity and elasticity, as well as in hydro- and gas-dynamics. Moreover, they also served as our model problems in our previous and closely related works~\cite{HeidWihler:2020a, HeidWihler:2020b, HeidPraetoriusWihler:2020,HeidWihlerm2an:2022}. 

To guarantee that the properties {\rm (F1)--(F5)} are satisfied in the given setting, we need to impose the following assumptions on the diffusion coefficient:
\begin{enumerate} 
\item[(M1)] The diffusion coefficient $\mu:\mathbb{R}_{\geq 0} \to \mathbb{R}_{\geq 0}$ is continuously differentiable and monotonically decreasing; i.e., $\mu'(t) \leq 0$ for all $t \geq 0$.
\item[(M2)] There exist positive constants $0<m_\mu<M_{\mu}<\infty$ such that
\begin{align} \label{eq:muprop}
m_\mu(t-s) \leq \mu(t^2)t-\mu(s^2)s\leq M_\mu (t-s) \qquad \text{for all} \ t \geq s \ge 0.
\end{align}
\end{enumerate}
The assumption~{\rm (M2)} implies {\rm (F1)} and {\rm (F2)} with $L=3 M_\mu$ and $\nu=m_\mu$, respectively; see, e.g.~\cite[Proposition 25.26]{Zeidler:90}. Furthermore, $\F$ is Gateaux differentiable and the derivative satisfies~\eqref{eq:Fpcoercive} and~\eqref{eq:Fpbounded} with $\alpha_{\F'}=m_\mu$ and $\beta_{\F'}=2M_\mu-m_\mu$, respectively; we refer to the proof of~\cite[Proposition 5.3]{HeidWihler:2020a}. Note that, in the given setting, $\nicefrac{2\alpha_{\F'}}{L}=\nicefrac{2m_\mu}{3 M_{\mu}}<1$, and thus the convergence of the classical Newton method is not guaranteed a priori. For the verification of (F4) we refer to~\cite[Lemma 5.4]{HeidWihler:2020a} and the proof of~\cite[Proposition 5.3]{HeidWihler:2020a}. Finally, $\F$ has the potential
\begin{align} \label{eq:Hdeff}
\H(u):=\int_\Omega \psi(|\nabla u|^2) \dx-\dprod{g,u},
\end{align}
where $\psi(s):=\frac{1}{2} \int_0^s \mu(t) \dt$.

\subsection{Numerical experiments}

Now we will run two experiments to test our adaptive algorithm. For that purpose, we will discretise the continuous problem with a conforming $P1$-finite element method. Moreover, we will compute an accurate approximation of the discrete solution with the Ka\v{c}anov iteration scheme, see, e.g.~\cite[\S4.5]{Necas:86} or~\cite[\S25.14]{Zeidler:90}, which is guaranteed to convergence given the assumption {\rm (M1)}. In Algorithm~\ref{alg:AdaptivedampedNewton}, we set the parameters $\sigma=0.8$ and $\theta=0.1$ for both of our experiments. 

\experiment \label{exp:1} In our first experiment, we consider the L-shaped domain $\Omega:=(-1,1)^2 \setminus ([0,1] \times [0,1])$ and the diffusion coefficient $\mu(t)=(t+1)^{-1}+\nicefrac{1}{2}$. It is straightforward to verify that $\mu$ satisfies {\rm (M1)} and {\rm (M2)}; specifically, it can be shown that $m_\mu = \nicefrac{3}{8}$ and $M_\mu=\nicefrac{3}{2}$. The source function $g$ is chosen in such a way that the exact solution of the continuous problem is given by $u^\star(x,y)=\sin(\pi x) \sin(\pi y)$, where $(x,y) \in \mathbb{R}^2$ denote the Euclidean coordinates. Finally, we consider the constant null function as our initial guess.

In this experiment, we compare the performance of Algorithm~\ref{alg:AdaptivedampedNewton} and the damped Newton method with fixed step-size $\delta=\nicefrac{\alpha_{\F'}}{L}=\nicefrac{m_\mu}{3 M_\mu}$. We can observe in Figure~\ref{fig:Exp1} that, in the given setting, $\delta \equiv 1$ is an admissible step-size and leads to a quadratic decay rate of the error. In contrast, the fixed damping parameter $\delta=\nicefrac{m_\mu}{3 M_\mu}$, which is chosen in accordance to Theorem~\ref{thm:convergence}, leads to a very poor convergence rate.

\begin{figure}
\centering
{\includegraphics[width=0.7\textwidth]{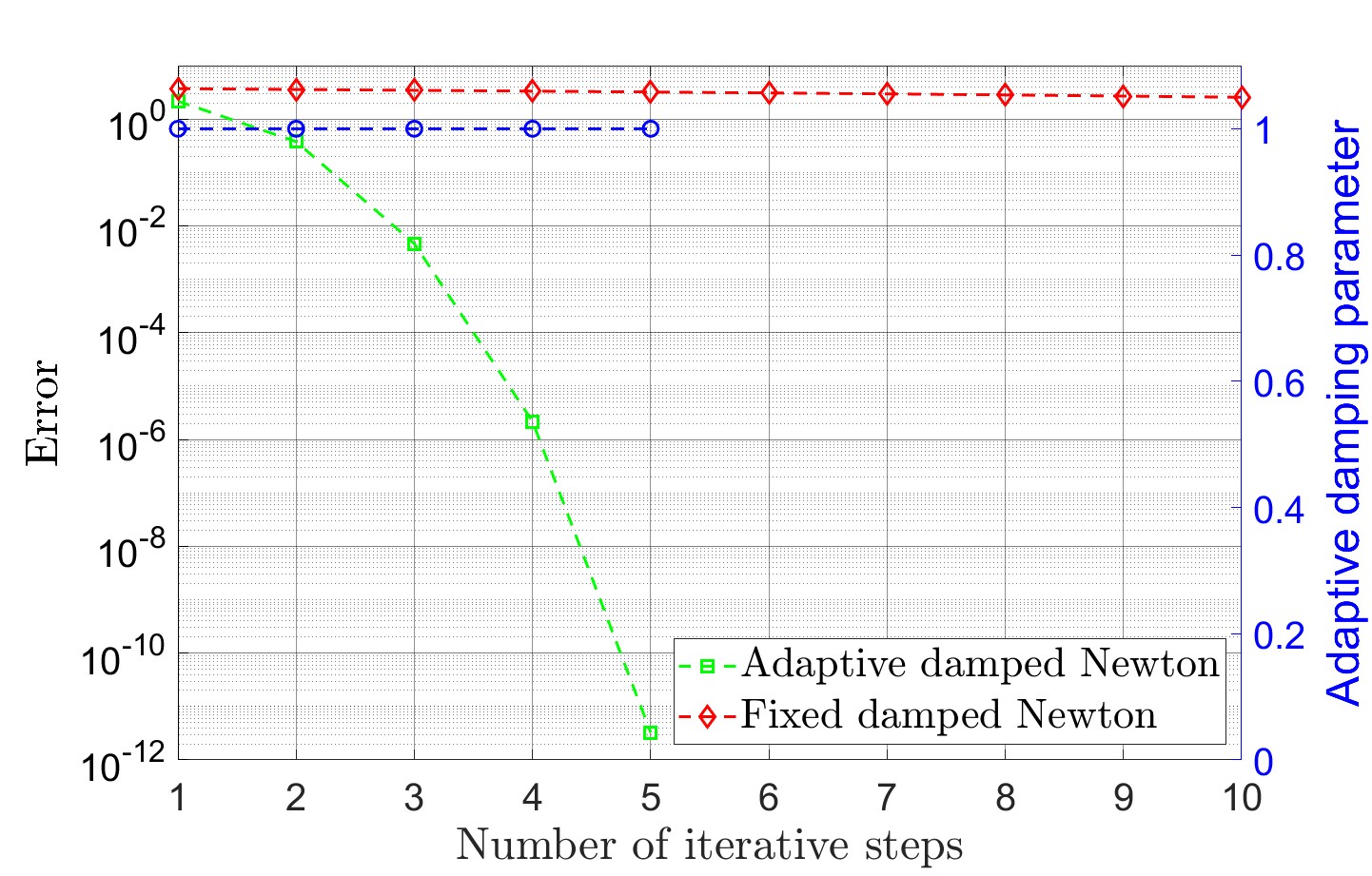}}
 \caption{Experiment~\ref{exp:1}: Comparison of the adaptive damped Newton method and the damped Newton method with fixed step-size.}\label{fig:Exp1}
\end{figure}

\experiment \label{exp:2} In our second numerical test, we consider the square domain $\Omega:=(0,1)^2$. Furthermore, our diffusion coefficient is given by
\begin{align} \label{eq:diffBE}
\mu(t):=\frac{\gamma}{\sqrt{t+k^{-2}}} +2\zeta,
\end{align} 
for some constants $k,\gamma,\zeta>0$; this function is based on the Bercovier--Engelman regularisation of the viscosity coefficient of a Bingham fluid, cf.~\cite{Bercovier:80}. Clearly, $\mu$ is differentiable and decreasing. Furthermore, the assumption {\rm (M2)} is satisfied with $m_\mu=2 \zeta$ and $M_\mu=2 \zeta + k \gamma$, see~\cite[Lemma 4.1]{HeidSuli2:21}. For our experiment, we will set $\gamma=0.3$, $\zeta=1$, and $k=100$. For simplicity, we will consider the same source function $g$ as in the previous experiment and set $u^0(x,y)=\sin(\pi x) \sin(\pi y)$, or more precisely its piecewise linear interpolation in the nodes of the mesh. 

Here, we consider the adaptive damped Newton method from Algorithm~\ref{alg:AdaptivedampedNewton} and the classical Newton scheme. As can be observed in Figure~\ref{fig:Exp2}, the undamped scheme does not converge in the given setting. In contrast, and in accordance with the theory, the adaptive step-size strategy from Algorithm~\ref{alg:AdaptivedampedNewton} leads to the convergence of the damped Newton method. Moreover, after a first phase of reduced step-sizes, and in turn of inferior decay rate, the damping parameter finally becomes one, which then leads to the local quadratic convergence rate. 

\begin{figure}
\centering
{\includegraphics[width=0.7\textwidth]{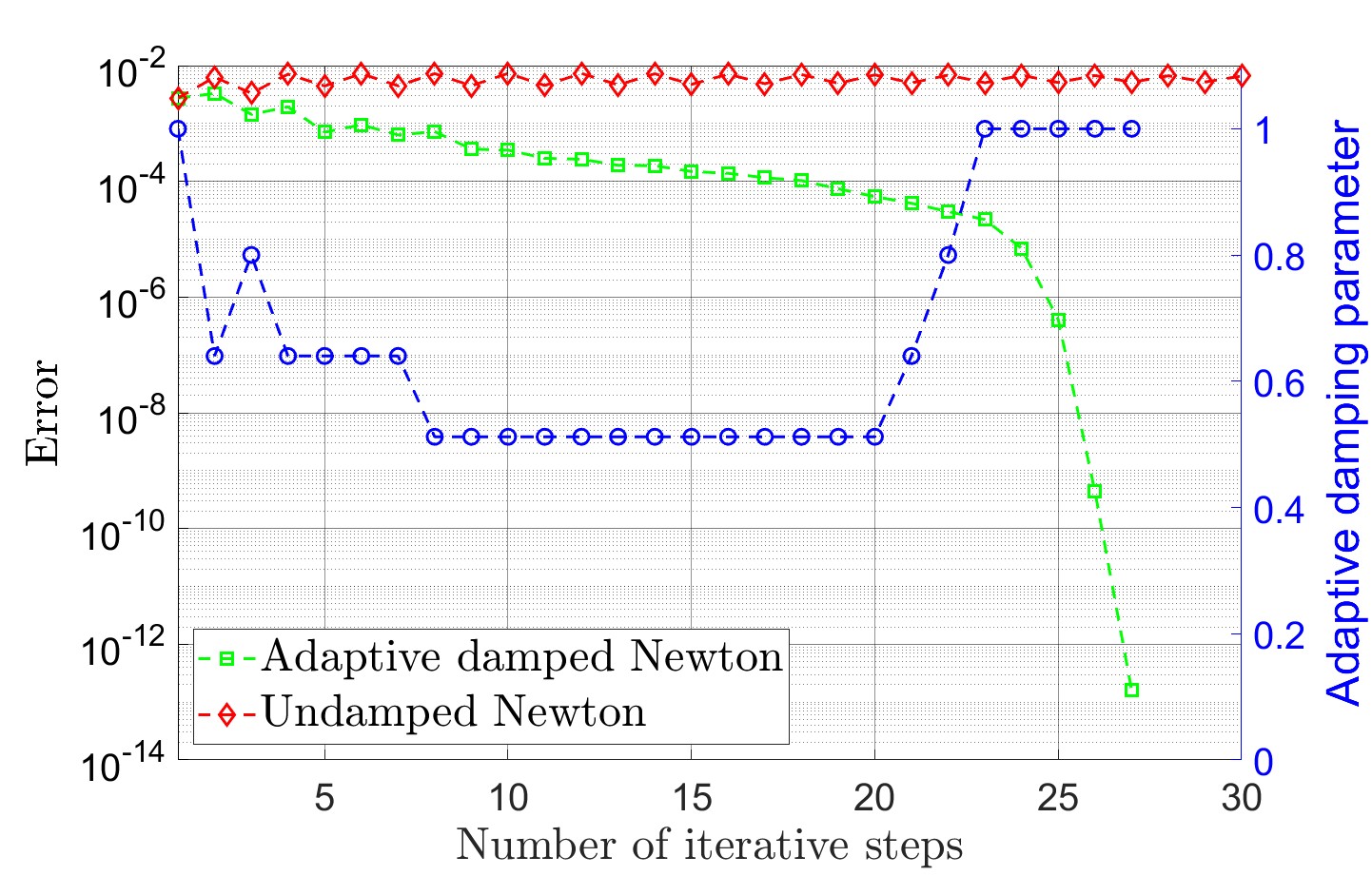}}
 \caption{Experiment~\ref{exp:2}: Comparison of the adaptive damped Newton method and the classical Newton method.}\label{fig:Exp2}
\end{figure}

\section{Conclusion}
We have presented an adaptively damped Newton method, which converges, under suitable assumptions, for any initial guess to the unique solution in the context of strongly monotone and Lipschitz continuous operator equations. Moreover, our numerical experiments highlighted that this algorithm may indeed converge in situations where the classical Newton scheme fails, and, nonetheless, exhibits the favourable local quadratic convergence rate of the undamped method.  

\bibliographystyle{amsalpha}
\bibliography{references}
\end{document}